\documentclass[11pt,english]{amsart}
\usepackage[latin9]{inputenc}
\usepackage[active]{srcltx}
\usepackage{babel}
\usepackage{amsthm}
\usepackage{amstext}
\usepackage{amssymb}
\usepackage{esint}
\usepackage{color}
\usepackage[unicode=true,pdfusetitle,
 bookmarks=true,bookmarksnumbered=false,bookmarksopen=false,
 breaklinks=false,pdfborder={0 0 1},colorlinks=false]
 {hyperref}

\makeatletter
\numberwithin{equation}{section}
\numberwithin{figure}{section}
\theoremstyle{plain}
\newtheorem{thm}{\protect\theoremname}
  \theoremstyle{plain}
  \newtheorem{lem}[thm]{\protect\lemmaname}
  \theoremstyle{plain}
  
  \theoremstyle{definition}
  \newtheorem{defn}[thm]{\protect\definitionname}

\usepackage{amsthm}\usepackage{amsfonts}\usepackage{amscd}\usepackage{url}\usepackage{color}\usepackage[matrix,arrow]{xy}\usepackage{mathrsfs}\usepackage{enumerate}\usepackage{amsopn}
\usepackage{bbm}
\usepackage{cite}\allowdisplaybreaks[1]
\usepackage{bbm}\usepackage{stmaryrd} \usepackage{mdef}
\date{\today}

  \providecommand{\definitionname}{Definition}
  \providecommand{\lemmaname}{Lemma}

\providecommand{\theoremname}{Theorem}

\makeatother

  \providecommand{\corollaryname}{Corollary}
  \providecommand{\definitionname}{Definition}
  \providecommand{\lemmaname}{Lemma}
\providecommand{\theoremname}{Theorem}

\begin{document}

\title[Semi-discretization for stochastic scalar conservation laws with multiple rough fluxes]{Semi-discretization for stochastic scalar conservation laws with multiple rough fluxes}
\begin{abstract}
\noindent We develop a semi-discretization approximation  for scalar conservation laws with multiple rough time dependence in inhomogeneous fluxes.  The method is based on Brenier's transport-collapse algorithm and uses characteristics defined in the setting of rough paths. 
We prove strong $L^1$-convergence for inhomogeneous fluxes and provide a rate of convergence for homogeneous one's. The approximation scheme as well as the proofs are based on the recently developed theory of pathwise entropy solutions and uses the kinetic formulation which allows to define globally the (rough) characteristics.
\end{abstract}

\author[B. Gess]{Benjamin Gess}
\address{Max-Planck Institute for Mathematics in the Sciences \\
04103 Leipzig\\
Germany }
\email{bgess@mis.mpg.de}

\author{Benoît Perthame}

\address{Sorbonne Universit\'es, UPMC Univ. Paris 06, CNRS, INRIA,  UMR 7598, Laboratoire Jacques-Louis Lions, \'Equipe MAMBA, 4, place Jussieu 75005, Paris, France.}

\email{benoit.perthame@upmc.fr}

\author{Panagiotis E. Souganidis }

\address{Department of Mathematics \\
 University of Chicago \\
 Chicago, IL 60637 \\
 USA }

\email{souganidis@math.uchicago.edu }



\maketitle

\section{Introduction}

We introduce a semi-discretization scheme and prove its convergence for stochastic scalar conservation laws (with multiple rough fluxes) of the form
\begin{equation}
\begin{cases}
du+ \displaystyle\sum_{i=1}^{N}\partial_{x_{i}}A^{i}(x,u)\circ dz^{i}_t =0\quad\text{in }\R^{N}\times(0,T)\label{eq:intro-SCL-inhomo},\\[1mm]
u(\cdot, 0)  =u_{0}\in(L^{1}\cap L^{2})(\R^{N}).
\end{cases}
\end{equation}
The precise assumptions on $A$, $z$ are presented in the sections \ref{sec:homo} and \ref{sec:inhomo} below. To introduce the results here we assume that 
$A\in C^{2}(\R^{N}\times\R;\R^{N})$ and  $z$ is an $\a$-Hölder geometric rough path; for example, $z$ may be a $d$-dimensional (fractional) Brownian motion or $z(t)=(t,\dots,t)$ in which case we are back in the classical deterministic setting --see Appendix \ref{sec:RP} for some background on rough paths. For spatially homogeneous fluxes, the theory is simpler and  $z\in C([0,T];\R^{N})$ is enough. In what follows we may occasionally use the term ``stochastic'' even when $z$ is a continuous or a rough path.

Stochastic scalar conservation laws of the type \eqref{eq:intro-SCL-inhomo} arise in several applications. For example, \eqref{eq:intro-SCL-inhomo} appears in the theory of mean field games developed by Lasry and Lions \cite{LL06}, \cite{LL06-2}, \cite{LL07}. We refer to Gess and Souganidis \cite{GS14} and Cardaliaguet, Delarue, Lasry and Lions \cite{CDLL15} for more details on the derivation of \eqref{eq:intro-SCL-inhomo} in this case.

The semi-discretization scheme we consider here is based on first rewriting \eqref{eq:intro-SCL-inhomo} in its kinetic form using the classical Maxwellian
\begin{equation}
\chi(x,\xi,t):=\chi(u(x,t),\xi):=\left\{ \begin{array}{l}
+1\text{ for }0\leq\xi\leq u(x,t),\\[2mm]
-1\text{ for }u(x,t)\leq\xi\leq0,\\[2mm]
\;0\text{ otherwise}.
\end{array}\right.\label{eq:char_fctn}
\end{equation}
The theory  of pathwise entropy solutions introduced by Lions, Perthame and Souganidis in \cite{LPS13} and further developed by Lions, Perthame and Souganidis in \cite{LPS14} and Gess and Souganidis  in \cite{GS14} (see  Appendix \ref{sec:recall_e_soln}) for the precise definition and some results) asserts that there exists a non-negative, bounded measure $m$ on $\R^N\times\R\times[0,T]$ such that, in the sense of distributions,
\begin{equation}
\partial_{t}\chi+\sum_{i=1}^{N}a^{i}(x,\xi)\partial_{x_{i}}\chi\circ dz^{i}+\sum_{i=1}^{N}\partial_{x_{i}}A^{i}(x,\xi)\partial_{\xi}\chi\circ dz^{i}=\partial_{\xi}m,\label{eq:kinetic-1-1}
\end{equation}
where, for notational simplicity we set   
  $$a^{i}(x,\xi):=(\partial_{u}A^{i})(x,\xi).$$
The approximation  is based on a splitting and fast relaxation scheme. Given a sequence of time steps $0=t_{0}<t_{1}<\dots<t_{K}=T$, 
we first solve the linear ``free-streaming'' transport equation
\begin{equation}
\partial_{t}f_{\D t}+\sum_{i=1}^{N}a^{i}(x,\xi)\partial_{x_{i}}f_{\D t}\circ dz^{i}+\sum_{i=1}^{N}\partial_{x_{i}}A^{i}(x,\xi)\partial_{\xi}f_{\D t}\circ dz^{i}=0\quad\text{on }\R^N\times\R\times[t_k,t_{k+1}),\label{eq:kinetic-1-1-1}
\end{equation}
and then introduce a fast relaxation step, setting (see section~\ref{section:notation} for the notation),
\begin{equation}\label{eq:uDt}
u_{\D t}(x,t):=\int f_{\D t}(x,\eta,t-)d\eta \ \  \ \text{and} \ \  \  f_{\D t}(x,\xi,t_{k+1}):=\chi(u_{\D t}(x,t_{k+1}),\xi);
\end{equation}
note that $f_{\D t}$ is discontinuous at $t_k$ while $u_{\D t}$ is not. 

The approximation of the pathwise entropy solution to \eqref{eq:intro-SCL-inhomo} that we are considering here is $u_{\D t}$. 

We present our results first for homogeneous, that is, $x$-independent fluxes, and then we treat the general case. 

Consider the homogeneous stochastic scalar conservation law
\begin{align}
du+\sum_{i=1}^{N}\partial_{x_{i}}A^{i}(u)\circ dz^{i} & =0\quad\text{in }\R^{N}\times(0,T).\label{eq:intro-SCL-homo}
\end{align}
We prove the strong convergence of the approximants $u_{\D t}$ to the pathwise entropy solution $u$ and provide an estimate for the rate of convergence (see Theorem \ref{thm:main} below), that is, for $u_{0}\in(BV\cap L^{\infty} \cap L^1)(\R^{N})$, we show that there exists $C>0$ depending only on the data such that 
\begin{equation}
\|u(t)-u_{\D t}(t)\|_{L^{1}}\le C\sqrt{\D z},\label{eq:intro-rate}
\end{equation}
where $\D z$ defined by
\begin{equation}\label{eq:Dz}
\D z:=\max_{k=0,\dots,K-1}\sup_{t\in[t_{k},t_{k+1}]}|z_{t}-z_{t_{k}}|.
\end{equation}
In the general inhomogeneous case, that is, for \eqref{eq:intro-SCL-inhomo}, no bounded variation estimates are known either for the solution $u$ or for the approximants $u_{\D t}$. In addition, due to the spatial dependence, we cannot use averaging techniques. To circumvent these difficulties, we devise a new method of proof based on the concept of generalized kinetic solutions and new energy estimates (see Lemma \ref{lem:tightness} below). The result (see Theorem \ref{thm:x-dep}) is that, for $u_{0}\in(L^{1}\cap L^{2})(\R^{N})$ and $\D t\to 0$, 
\[
u_{\D t}\to u\quad\text{in }L^{1}(\R^{N}\times[0,T]).
\]

The semi-discretization scheme we introduce here is a generalization of the transport-collapse scheme developed by Brenier \cite{B83,B84} and Giga and Miyakawa \cite{GM83} for the deterministic homogeneous scalar conservation law
\begin{align}
\partial_{t}u+\sum_{i=1}^{N}\partial_{x_{i}}A^{i}(u) & =0\quad\text{in }\R^{N}\times(0,T).\label{eq:intro-SCL-det}
\end{align}
In this setting, the convergence of the transport-collapse scheme was proven in \cite{B83,B84,GM83} based on bounded variation arguments. A general methodology for this type of result as well as for error estimates was developed by Bouchut and Perthame \cite{BP98}. An alternative proof of the weak convergence of the transport-collapse scheme based on averaging techniques was presented by Vasseur in \cite{V99} for \eqref{eq:intro-SCL-det} with $N=1$ and $A^{i}(u)=\frac{1}{2}u^{2}$, that is, for Burgers' equation. 

The results we present here generalize what was known before even for the deterministic problem. Indeed, firstly, we establish a rate of convergence for the transport-collapse scheme (see~\eqref{eq:intro-rate}) which was previously unavailable even in the deterministic case (although maybe not too surprising in view of \cite{BP98}). Secondly, we prove the convergence of the 
scheme also in the inhomogeneous case, where the averaging techniques and, thus, the method developed in \cite{V99} do not apply, in particular because our assumptions allow for degenerate fluxes.

The well-posedness of the pathwise entropy solutions for \eqref{eq:intro-SCL-inhomo} has been proven in \cite{LPS13,LPS14,GS14}. Regularity and long-time behavior has been considered by Lions, Perthame and Souganidis \cite{LPS12} and Gess and Souganidis \cite{GS14-2}. For a detailed account of numerical methods for (deterministic) conservation laws we refer to LeVeque \cite{L92}, Bouchut \cite{B04}, Godlewski and Raviart \cite{GR96}, Eymard and Gallou{\"e}t, Herbin \cite{EGH00} and the references therein.

Finally, we recall that kinetic solutions to \eqref{eq:intro-SCL-det} were constructed by Brenier and Corrias \cite{BC98}, Lions, Perthame and Tadmor \cite{LPT94} and Perthame \cite{P02} as limits of the so-called Bhatnagar, Gross, Krook (BGK) approximation, that is, 
\begin{equation}
\partial_{t}f^{\ve}+\sum_{i=1}^{N}(A^{i})'(\xi)\partial_{x_{i}}f^{\ve}=\frac{1}{\ve}(\mcM f^{\ve}-f^{\ve}),\label{eq:intro-BDK}
\end{equation}
where  the ``Maxwellian'' associated with a distribution $f$ is defined by 
\begin{equation}
\mcM f(x,\xi,t):=\chi(\int f(x,\eta,t)d\eta,\xi). \label{eq:intro-maxwellian}
\end{equation}
In comparison, the transport-collapse scheme we are considering here is based on a fast relaxation scale for the right-hand side of \eqref{eq:intro-BDK}, that is on enforcing $\mcM f^{\ve}=f^{\ve}$ at the time-steps $t_{k}$. 

\subsection*{Structure of the paper}

The strong convergence and the rate for the homogeneous case is obtained in section \ref{sec:homo}. The inhomogeneous case is treated in section \ref{sec:inhomo}. Some background for the theory of rough paths is presented in Appendix \ref{sec:RP}. The definition and fundamental properties of pathwise entropy solutions to \eqref{eq:intro-SCL-inhomo} are recalled in Appendix \ref{sec:recall_e_soln}. A basic, but crucial, bounded variation estimate for indicator functions is given in Appendix \ref{sec:indicator}.

\subsection*{Notation}
\label{section:notation}
We set $\R_+:=(0,\infty)$ and $\delta$ is the ``Dirac'' mass at the origin in $\R$. The complement and closure of a set $A\subseteq \R^N$ are  denoted respectively by $A^c$ and $\bar A$,  and $B_R$ is the open ball in $\R^N$ centered at the origin with radius $R$.  We write  $\|f\|_{C(\mcO)}$  for the sup norm of a  continuous bounded function $f$ on $\mcO\subseteq\R^{M}$ and, for $k=1,\ldots,\infty$,  we let  $C_{c}^{k}(\mcO)$  be the space of all $k$ times continuously differentiable functions with compact support in $\mcO$. For $\gamma>0$, $\text{Lip}^\gamma(\mcO;\R^{l})$
is the set of $\R^l$ valued functions defined on $\mcO$ with $k=0, \ldots\lfloor\gamma\rfloor$ bounded derivatives and $\gamma-\lfloor\delta\rfloor$ Hölder continuous $\lfloor\gamma\rfloor$-th derivative; for simplicity, if $\gamma=1$ and $l=1$, we write $\text{Lip}(\mcO)$ and denote by $\|\cdot\|_{C^{0,1}}$ the Lipschitz constant.  The subspace of $L^{1}$-functions with bounded total variation is $BV$. If $f \in BV$, then $\|f\|_{BV}$ is its total variation. 
For $u\in L^1([0,T]; L^p(\R^N))$ we write $\|u(t)\|_p$ for the $L^p$ norm of $u(\cdot,t)$. To simplify the presentation, given a function $f(x,\xi)$ we write $\|f\|_{L^1_{x,\xi}}:=\int |f|dxd\xi:=\int |f(x,\xi)|dxd\xi$. For a measure $m$ on $\R^N\times\R\times[0,T]$ we often write $m(x,\xi,t)dxd\xi dt$ instead of $dm(x,\xi,t)$. If $f\in L^{1}(\R^{N}\times[0,T])$ is such that $t\mapsto f(\cdot,t)\in L^{1}(\R^{N})$ is càdlàg, that is, right-continuous with left limits, we let 
\[
\int f(x,t-)dx:=\lim_{h\downarrow0}\int f(x,t-h)dx.
\]
The space of all càdlàg functions from an interval $[0,T]$ to a metric space $M$ is denoted by $D([0,T];M)$. For a function $f:[0,T] \to \R$ and $a,b\in [0,T]$ we set $f|_a^b := f(b)-f(a)$. 
The negative and positive part of a function $f:\R^N\to\R$ are defined by $f^- := \max\{-f,0\}$ and $f^+ := \max\{f,0\}$. Finally, given $a,b \in \R$, $a\wedge b:=min(a,b).$

\section{Spatially homogeneous stochastic scalar conservation laws\label{sec:homo}}

We consider stochastic homogeneous scalar conservation laws, that is, the initial value problem
\begin{equation}\label{eq:SCL}
\begin{cases}
du+\sum_{i=1}^{N}\partial_{x_{i}}A^{i}(u)\circ dz^{i}  =0\quad\text{in }\R^{N}\times(0,T),\\[1mm]
u(\cdot, 0)  =u_{0}\in(BV\cap L^{\infty})(\R^{N}), 
\end{cases}
\end{equation}
where
\begin{equation}\label{eq:homo_assumptions}
   z\in C([0,T];\R^{N})\quad\text{and}\quad  A\in C^{2}(\R;\R^{N});
\end{equation}
recall that as mentioned earlier $BV$ is taken to be a subset of $L^1$. 

The kinetic formulation reads, informally,
\begin{equation}
\partial_{t}\chi+\sum_{i=1}^{N}a^{i}(\xi)\partial_{x_{i}}\chi\circ dz^{i}=\partial_{\xi}m,\label{eq:kinetic}
\end{equation}
for some non-negative, bounded measure $m$ on $\R^N\times\R\times [0,T]$, where $a:=A'.$

Fix $\D t>0$, define $t_{k}:=k\D t$ with $k=0,\dots,K$ and $K\D t \approx T$ and $\D z$ as in \eqref{eq:Dz}. 
In what follows assume 
\begin{equation}\label{eq:Dz_ass}
  \D z\le1.
\end{equation}
The approximation $u_{\D t}$ is defined as 
\begin{equation}\label{eq:uD-def}  
u_{\D t}(x,0)=u_{0}(x) \ \ \ \text{and} \ \  \   u_{\D t}(x,t):=\int f_{\D t}(x,\xi,t-)d\xi,
\end{equation}
where $f_{\D t}$ is the solution of
\begin{equation*}
\begin{cases}
\partial_{t}f_{\D t}+\sum_{i=1}^{N}a^{i}(\xi)\partial_{x_{i}}f_{\D t}\circ dz^{i}  =0\quad\text{on }(t_{k},t_{k+1})\\[1mm]
f_{\D t}(x,\xi,t_{k})  =\chi(u_{\D t}(x,t_{k}),\xi),
\end{cases}
\end{equation*}
that is, for $t\in[t_{k},t_{k+1})$,
\begin{align}
f_{\D t}(x,\xi,t) & =f_{\D t}(x-a(\xi)(z_{t}-z_{t_{k}}),\xi,t_{k}).\label{eq:scheme}
\end{align}

The main result in this section is:
\begin{thm}
\label{thm:main}Let $u_{0}\in(BV\cap L^{\infty})(\R^{N})$ and assume \eqref{eq:homo_assumptions} and \eqref{eq:Dz_ass}. Then
\[
\sup_{t\in[0,T]}\|u(\cdot,t)-u_{\D t}(\cdot,t)\|_1\le \sqrt{ 2\|u_{0}\|_{BV}  \|a\|_{C^{0,1}([-\|u_0\|, \|u_0\|)} } \, \|u_{0}\|_{2}\sqrt{\D z}.
\]
\end{thm}

Before presenting the rigorous proof of Theorem \ref{thm:main} we give an informal overview of the argument. For the sake of this exposition we assume $z\in C^1([0,T];\R^N)$ for now. 

The proof is based on the observation that the semi-discretization scheme introduced above has a kinetic interpretation. Indeed, using the notation 
\eqref{eq:intro-maxwellian}, we observe that $f_{\D t}$ solves
\begin{align}
\partial_{t}f_{\D t}+\sum_{i=1}^{N}a^{i}(\xi)\partial_{x_{i}}f_{\D t}\dot{z}^{i} & =\sum_{k}\d(t-t_{k})(\mcM f_{\D t}-f_{\D t})\label{eq:approx_kinetic-2}
 =:\partial_{\xi}m_{\D t}. 
\end{align}
Since $|\chi|=|f_{\D t}|=1$ or $0$, it follows that 
\begin{align*}
\int|\chi(t)-f_{\D t}(t)|d\xi dx & =\int|\chi(t)-f_{\D t}(t)|^{2}d\xi dx
  =\int\left(|\chi(t)|^{2}-2\chi(t)f_{\D t}(t)+|f_{\D t}(t)|^{2}\right)d\xi dx\\
 & =\int\left(|\chi(t)|-2\chi(t)f_{\D t}(t)+|f_{\D t}(t)|\right)d\xi dx.
\end{align*}
Multiplying \eqref{eq:kinetic-1-1-1} and \eqref{eq:approx_kinetic-2} by $\sgn(\xi)$ and integrating yields
\begin{align*}
\frac{d}{dt}\int|\chi(t)|d\xi dx  = -2 \int m(x,0,t)dx\quad\text{and}\quad
\frac{d}{dt}\int|f_{\D t}(t)|d\xi dx  =-2 \int m_{\D t}(x,0,t)dx,
\end{align*}
and, since
$$\partial_{\xi}\chi =\d(\xi)-\d(u(x,t)-\xi) \le \d(\xi),$$
and 
$$\partial_{\xi}f_{\D t}  \le \d(\xi)+D_{x}f_{\D t}(x-a(\xi)(z_{t}-z_{t_{k}}),\xi,t_{k})\cdot a'(\xi)(z_{t}-z_{t_{k}}),$$
we obtain 
\begin{align*}
-2\frac{d}{dt}\int\chi f_{\D t}d\xi dx & =-2\int\left(\partial_{t}\chi f_{\D t}+\chi\partial_{t}f_{\D t}\right)d\xi dx\\
 & =-2\int \left(f_{\D t}\big(-\sum_{i=1}^{N}a^{i}(\xi)\partial_{x_{i}}\chi\dot{z}^{i}+\partial_{\xi}m\big)+\chi\big(-\sum_{i=1}^{N}a^{i}(\xi)\partial_{x_{i}}f_{\D t}\dot{z}^{i}+\partial_{\xi}m{}_{\D t}\big)\right)d\xi dx\\
 & =2\int\left(\partial_{\xi}f_{\D t}m+\partial_{\xi}\chi m{}_{\D t}\right)d\xi dx\\
 & \le2\int\left( m(x,0,t)+m{}_{\D t}(x,0,t)\right)d\xi dx\\
 & \quad +\int D_{x}f_{\D t}(x-a(\xi)(z_{t}-z_{t_{k}}),\xi,t_{k})\cdot a'(\xi)(z_{t}-z_{t_{k}})md\xi dx\\
 & \le-\frac{d}{dt}\int|\chi|d\xi dx-\frac{d}{dt}\int|f_{\D t}|d\xi dx+\|a'\|_{C^{0}([-\eta,\eta])}|z_{t}-z_{t_{k}}|\int|D_{x}m|d\xi dx,
\end{align*}
and, hence,
\begin{align*}
\frac{d}{dt}\int|\chi(t)-f_{\D t}(t)|d\xi dx & \le\|a'\|_{C^{0}([-\eta,\eta])}|z_{t}-z_{t_{k}}|\int|D_{x}m|d\xi dx.
\end{align*}

At this point we face a difficulty. The term $\int|D_{x}m|d\xi dx$ may not be finite and thus an additional approximation argument is necessary. 

To resolve this issue  we replace $\chi$ by its space mollification $\chi^{\ve}$ making an error of order  $\ve \|u_{0}\|_{BV}$ and we note that, if $m^{\ve}$ is the mollification of $m$ with respect to the $x$-variable,
\[
\int_{0}^{T}\int|D_{x}m^{\ve}|d\xi dxdt\le \frac{1}{\ve}  \int_{0}^{T}\int m d\xi dxdt \leq \frac{\|u_{0}\|_{2}}{2\ve}.
\]
In conclusion, we find
\[
\int|u(t)-u_{\D t}(t)|dx\lesssim \ve \|u_{0}\|_{BV} +\|a'\|_{C^{0}([-\eta,\eta])}\D z\frac{\|u_{0}\|_{2}}{2\ve},
\]
and choosing $\ve=\sqrt{\D z}$ finishes the informal  proof.
\smallskip

For future reference we observe that, if 
\[
\chi_{\D t}(x,\xi,t):=\chi(u_{\D t}(x,t),\xi), 
\]
then
\begin{equation}\label{eq:mcM} 
 \chi_{\D t}(x,\xi,t)= \chi(\int f_{\D t}(x,\eta,t)d\eta,\xi)= \mcM f_{\D t}(x,\xi,t), 
\end{equation}
and, to simplify the notation, we set 
 $$\eta:=\|u_{0}\|_{\infty}.$$
We continue with
\begin{proof}[The proof of Theorem \ref{thm:main}]

We first assume $z\in C^{1}([0,T];\R^{N})$. In this case, $\chi$ and $f_{\D t}$ solve \eqref{eq:kinetic} and \eqref{eq:approx_kinetic-2} respectively. It has been shown in Theorem 3.2 in \cite{LPS13} that $\chi$ depends continuously on the driving signal $z$, in the sense that, if  $u^{1},u^{2}$ are two solutions  driven by $z^{1},z^{2}$ respectively, then 
\[
\sup_{t\in[0,T]}\|u^{1}(t)-u^{2}(t)\|_1 \le C\|z^{1}-z^{2}\|_{C([0,T];\R^{N})}.
\]
In view of \eqref{eq:scheme}, it   follows  that  $f_{\D t}$ and $\chi_{\D t}$ also 
depend continuously on $z$. Hence, the rough case $z\in C([0,T];\R^{N})$ can be handled by smooth approximation in the end (see step 5 below). 

\textbf{\textit{Step 1: The kinetic formulation. }}As mentioned earlier, the proof is based on the kinetic interpretation of the semi-discretization scheme given by \eqref{eq:approx_kinetic-2}.

Since, in view of \eqref{eq:mcM}
\[
\int(\mcM f)(x,\xi)d\xi=\int f(x,\xi)d\xi,
\]
we use 
\begin{align*}
\int|\mcM f(x,\xi)-\mcM g(x,\xi)|d\xi & =\int \big| \chi(\int f(x,\eta)d\eta,\xi)-\chi(\int g(x,\eta)d\eta,\xi) \big|d\xi
 =|\int f(x,\xi)-g(x,\xi)d\xi |,
\end{align*}
to conclude the following $L^{1}-$ contraction property
\begin{equation}
\|\mcM f-\mcM g\|_{L_{x,\xi}^{1}}\le\|f-g\|_{L_{x,\xi}^{1}}.\label{eq:mcM-contraction}
\end{equation}
We note that $m_{\D t}$ is a non-negative measure. Indeed, 
\begin{align*}
m_{\D t} & =\int_{0}^{\xi}\sum_{k}\d(t-t_{k})(\mcM f_{\D t}-f_{\D t})d\td\xi
 =\sum_{k}\d(t-t_{k})\int_{0}^{\xi}(\mcM f_{\D t}-f_{\D t})d\td\xi
\end{align*}
and, moreover,
\begin{align*}
\int_{0}^{\xi}\mcM f_{\D t}(t)d\td\xi 
& =\int_{0}^{\xi}\chi(\int f_{\D t}(x,\eta,t)d\eta,\td\xi)d\td\xi
 =\xi\wedge\int f_{\D t}(x,\eta,t)d\eta
\end{align*}
Since $f_{\D t}\le1$ we find
\begin{align*}
\int_{0}^{\xi}f_{\D t}(t)d\td\xi & \le\xi\wedge\int f_{\D t}(t)d\td\xi,
\end{align*}
and, hence, 
\[
\int_{0}^{\xi}(\mcM f_{\D t}-f_{\D t})d\td\xi\ge0.
\]

\textbf{\textit{Step 2: The approximation.}} We continue with the uniqueness argument  introduced by Perthame in \cite{P98, P02} as an alternative to Kru{\v{z}}kov's method 
which is well-adapted to the kinetic formulation. 

Aiming to estimate the error
\begin{align*}
\int|u(t)-u_{\D t}(t)|dx &  = \int\big| \int ( \chi(t)- f_{\D t}(t)) d\xi \big|dx \leq \int|\chi(t)- f_{\D t}(t)|d\xi dx,
\end{align*}
we begin by regularizing $\chi$ using a standard Dirac sequence $\vp^{\ve}=\frac{1}{\ve^N}\vp (\frac x \ve)$, with $\| \vp \|_1=1$. That is, we consider the $x$-convolution (the rigorous proof uses also regularization in time so that the equation on $\chi^{\ve}$ is satisfied in a classical way,  but this technicality does not play a role here),
\[
\chi^{\ve}(x,\xi,t):=(\chi(\cdot,\xi,t)\ast\vp^{\ve})(x),
\]
which solves, for $m^\ve = m \ast \vp^\ve$,
\[
\partial_{t}\chi^{\ve}+\sum_{i=1}^{N}a^{i}(\xi)\partial_{x_{i}}\chi^{\ve}\dot{z}^{i}=\partial_{\xi}m^{\ve}.
\]
We first note that
\begin{align} \label{eq:initial_split}
 \int|\chi(t)-f_{\D t}(t)|d\xi dx & =\int|\chi(t)-f_{\D t}(t)|^{2}d\xi dx\nonumber \\
  & =\int|\chi(t)|-2\chi(t) f_{\D t}(t)+|f_{\D t}(t)|d\xi dx\\
 & =F^{\ve}(t) +Err^{1}(t),\nonumber 
\end{align}
where
\begin{equation*}F^{\ve}(t):=\int\left(|\chi^{\ve}(t)|-2\chi^{\ve}(t)f_{\D t}(t)+|f_{\D t}(t)|\right)d\xi dx, 
\end{equation*}
and
\begin{equation*}
Err^{1}(t) :=\int\left(|\chi(t)|-|\chi^{\ve}(t)|-2(\chi(t)-\chi^{\ve}(t))f_{\D t}(t)\right)d\xi dx.
\end{equation*}
Since $u(\cdot,0)=u_{\Delta t}(\cdot,0)$, it follows that 
\begin{equation}\label{eq:split_up}
\begin{cases}
\int|\chi(t)-f_{\D t}(t)|d\xi dx  =\int|\chi(t)-f_{\D t}(t)|d\xi dx-\int|\chi(0)-\chi_{\D t}(0)|d\xi dx\\[1.2mm]
\qquad\qquad  = \int_0^t \frac{d}{dt} F^{\ve}(s) ds
   +Err^{1}|_{0}^{t}.
\end{cases}
\end{equation}

\textbf{\textit{Step 3: The estimate of $\frac{d}{dt}F^{\ve}$}.} 
Using the identity
\[
\sgn(f_{\D t}(x,\xi,t))=\sgn(\xi), 
\]
we first note that 
\begin{align}\label{eq:l1-der}
\frac{d}{dt}\int|\chi^{\ve}|d\xi dx  =-2\int m^{\ve}(x,0,t)dx\quad\text{and}\quad
\frac{d}{dt}\int|f_{\D t}|d\xi dx  =-2\int m_{\D t}(x,0,t)dx.
\end{align}
Furthermore, since, in the sense of distributions,
\begin{align}
\partial_{\xi}\chi^{\ve} & =(\d(\xi)-\d(\xi-u(x,t)))\ast\vp^{\ve}\label{eq:xi_der} \le\d(\xi) \; \text{and }  \partial_{\xi} \chi \leq \d(\xi), 
\end{align}
and $f_{\D t}(x,\xi,t_{k})=\chi(u(x,t_{k}),\xi)$, for $t\in[t_k, t_{k+1})$, we find
\begin{align}
  \partial_{\xi}f_{\D t}\nonumber 
 & =\partial_{\xi} \big(f_{\D t}(x-a(\xi)(z_{t}-z_{t_{k}}),\xi,t_{k}) \big)\nonumber \\[1mm]
 & =(\partial_{\xi}\chi_{\D t})(x-a(\xi)(z_{t}-z_{t_{k}}),\xi,t_{k})+D_{x}f_{\D t}(x-a(\xi)(z_{t}-z_{t_{k}}),\xi,t_{k})\cdot a'(\xi)(z_{t}-z_{t_{k}})\label{eq:xi-der_1}\\[1mm]
 & \le \d(\xi)+D_{x}f_{\D t}(x-a(\xi)(z_{t}-z_{t_{k}}),\xi,t_{k})\cdot a'(\xi)(z_{t}-z_{t_{k}})\nonumber \\[1mm]
 & \le\d(\xi)+D_{x}f_{\D t}(x-a(\xi)(z_{t}-z_{t_{k}}),\xi,t_{k})\cdot a'(\xi)(z_{t}-z_{t_{k}}).\nonumber 
\end{align}
Using next \eqref{eq:l1-der}, \eqref{eq:xi_der}, \eqref{eq:xi-der_1} and $|f_{\D t}|\le1$ we obtain
\begin{align*}
-2&\frac{d}{dt}\int
 \chi^{\ve}f_{\D t}d\xi dx  =-2\int\partial_{t}\chi^{\ve}f_{\D t}+\chi^{\ve}\partial_{t}f_{\D t}d\xi dx\\
 & =-2\int \left(f_{\D t}\big(-\sum_{i=1}^{N}a^{i}(\xi)\partial_{x_{i}}\chi^{\ve}\dot{z}^{i}+\partial_{\xi}m^{\ve}\big) +\chi^{\ve}\big(-\sum_{i=1}^{N}a^{i}(\xi)\partial_{x_{i}}f_{\D t}\dot{z}^{i}+\partial_{\xi}m{}_{\D t}\big)\right)d\xi dx\\
 & =2\int\left(\partial_{\xi}f_{\D t}m^{\ve}+\partial_{\xi}\chi^{\ve}m{}_{\D t}\right)d\xi dx\\
 &\le2\int \left(m^{\ve}(x,0,t)+m{}_{\D t}(x,0,t)\right)d\xi dx \\
 & \quad +  2\sum_k  {1\hspace{-1.2mm}{\rm I}}_{\{t_k < t < t_{k+1}\} }   \int D_{x}f_{\D t}(x-a(\xi)\D z,\xi,t_{k})\cdot a'(\xi)(z_{t}-z_{t_{k}})m^{\ve}d\xi dx\\
 & \le-\frac{d}{dt}\int|\chi^{\ve}|d\xi dx-\frac{d}{dt}\int|f_{\D t}|d\xi dx+2\|a'\|_{C^{0}([-\eta,\eta])}\D z \int|D_{x}m^{\ve}|d\xi dx\\[1mm]
 & \le-\frac{d}{dt}\int|\chi^{\ve}|d\xi dx-\frac{d}{dt}\int|f_{\D t}|d\xi dx+2\frac{\|a'\|_{C^{0}([-\eta,\eta])}\D z}{\ve} \int m(x,\xi,t)d\xi dx,
\end{align*}
and, in conclusion, 
\begin{equation}\label{eq:F_ineq}
   \frac{d}{dt} F^{\ve}(t)\le 2 \frac{\|a'\|_{C^{0}([-\eta,\eta])} \D z}{\ve} \int m(x,\xi,t)d\xi dx.
\end{equation}

\textbf{\textit{Step 4: The estimate of $Err^{1}$.}} We estimate $|Err^{1}(t)|$ in terms of the $BV$-norm of $u_0$. Since $f_{\D t} \in \{0,\pm1\}$, we first observe that, for all $t\ge 0$, 
\[
|Err^{1}(t)|\le  \int|\chi(t)-\chi^{\ve}(t)|d\xi dx,
\]
and we state and prove the next lemma.
\begin{lem}\label{lem:err1} Assume  $u_{0}\in(BV\cap L^{\infty})(\R^{N})$ and \eqref{eq:homo_assumptions}. Then,
\[
\int|\chi(t)-\chi^{\ve}(t)|d\xi dx\le\ve\|u_{0}\|_{BV},
\]
and, for all $t\in[0,T]$,
\[
|Err^{1}(t)|\le \ve\|u_{0}\|_{BV} .
\]
\end{lem}
\begin{proof}
The $\chi^{\ve}$'s do not only take the values $0$ and $\pm 1$ but  instead $|\chi^{\ve}| \leq 1$. However, since $\chi =0$ or $1$ if $\xi \geq 0$ and $\chi =0$ or $-1$  if $\xi \leq 0$, it follows that  $\chi^{\ve} \geq 0$ if $\xi \geq 0$ and  $\chi^{\ve} \leq 0$ if $\xi \leq 0$. 

Hence, 
\begin{align*}
\int|\chi(t)-\chi^{\ve}(t)|d\xi dx & 
= \int\left(|\chi(t)|-2\chi(t)\chi^{\ve}(t)+|\chi^{\ve}(t)|\right)d\xi dx
\end{align*}
and 
\begin{align*}
-2\frac{d}{dt}\int\chi\chi^{\ve}d\xi dx & =-2\int\left(\big(-\sum_{i=1}^{N}a^{i}(\xi)\partial_{x_{i}}\chi\dot{z}^{i}+\partial_{\xi}m\big)\chi^{\ve}+\chi\big(-\sum_{i=1}^{N}a^{i}(\xi)\partial_{x_{i}}\chi^{\ve}\dot{z}^{i}+\partial_{\xi}m^{\ve}\big)\right)d\xi dx\\
 & =2\int \left(m\partial_{\xi}\chi^{\ve}+\partial_{\xi}\chi m^{\ve}\right)d\xi dx \le2\int \left(m(x,0,t)+m^{\ve}(x,0,t)\right)d\xi dx\\
 & =-\frac{d}{dt}\int|\chi|d\xi dx-\frac{d}{dt}\int|\chi^{\ve}|d\xi dx,
\end{align*}
and, thus,
\begin{align*}
\frac{d}{dt}\int|\chi-\chi^{\ve}|d\xi dx & \le0.
\end{align*}
Since $u_{0}\in BV$, using Lemma \ref{lem:BV-indicator}, we find 
\begin{align*}
\int|\chi(t)-\chi^{\ve}(t)|d\xi dx & \le\int|\chi(0)-\chi^{\ve}(0)|d\xi dx \le\ve\int\|\chi(\cdot,\xi,0)\|_{BV}d\xi  =\ve\|u_{0}\|_{BV}.
\end{align*}
\end{proof}

\textbf{\textit{Step 5: The conclusion.}} It follows from \eqref{eq:split_up}, \eqref{eq:F_ineq} and Lemma \ref{lem:err1} that, for all $t\in[0,T]$,
\begin{align*}
\int|\chi(x,\xi,t)- f_{\D t}(x,\xi,t)|d\xi dx & \le 2 \frac{\|a'\|_{C^{0}([-\eta,\eta])}  \D z}{\ve}\int _0^t\int m(x,\xi,r)dxd\xi dr + 2 \ve \|u_{0}\|_{BV} \\
& \le\frac{\|a'\|_{C^{0}([-\eta,\eta])}  \D z}{\ve}\|u_0\|_2^2+ 2 \ve \|u_{0}\|_{BV},
\end{align*}
and hence, choosing $\ve \approx \sqrt{\D z}$  to minimize the expression yields
\begin{align}
\int|\chi(x,\xi,t)-\chi_{\D t}(x,\xi,t)|d\xi dx & \le  \sqrt{ 2\|u_{0}\|_{BV}  \|a\|_{C^{0,1}([-\eta,\eta])} } \, \|u_{0}\|_{2}\, \sqrt{\D z}.\label{eq:final_1}
\end{align}

We now go back to $z\in C([0,T];\R^{N})$ and choose $z^{n}\in C^{1}([0,T];\R^{N})$ such that $z^{n}\to z$ in $C([0,T];\R^{N})$. In view of the continuity in the driving signal, we observe that, as $n\to\infty$
\begin{align*}
\chi^{n}  \to\chi\quad\text{and}\quad
\chi_{\D t}^{n}  \to\chi_{\D t}\quad\text{in }\ C([0,T];L^{1}(\R^{N+1})).
\end{align*}
It follows from \eqref{eq:final_1} that
\begin{align*}
\int|\chi^{n}(x,\xi,t)-\chi_{\D t}^{n}(x,\xi,t)|d\xi dx & \le
 \sqrt{ 2\|u_{0}\|_{BV}  \|a\|_{C^{0, 1}([-\eta,\eta])} } \, \|u_{0}\|_{2}\, \sqrt{\D z^n}
\end{align*}
Passing to the limit in $n$ completes the proof.
\end{proof}

\section{Spatially inhomogeneous stochastic scalar conservation laws\label{sec:inhomo}}

We consider  here the inhomogeneous stochastic scalar conservation law
\begin{equation} \label{eq:SCL-inhomo}
\begin{cases}
\partial_{t}u+\sum_{i=1}^{N}\partial_{x_{i}}A^{i}(x,u)\circ dz^{i}  =0\quad\text{in }\R^{N}\times(0,T),\\ 
u(\cdot, 0)  =u_{0}\in(L^{1}\cap L^{2})(\R^{N}), 
\end{cases}
\end{equation}
and its kinetic formulation
\begin{equation}
\partial_{t}\chi+\sum_{i=1}^{N}a^{i}(x,\xi)\partial_{x_{i}}\chi\circ dz^{i}+\sum_{i=1}^{N}\partial_{x_{i}}A^{i}(x,\xi)\partial_{\xi}\chi\circ dz^{i}=\partial_{\xi}m,\label{eq:kinetic-1}
\end{equation}
where 
$$a^{i}(x,\xi):=(\partial_{u}A^{i})(x,\xi)\quad\text{and}\quad b(x,\xi):=(\partial_{x_{i}}A^{i}(x,\xi))_{i=1}^{N}$$ 
and $z$ is an $\a$-Hölder geometric rough path for some $\a\in(0,1)$.  

More precisely, we assume that
\begin{equation}\label{eq:inhomo_assumptions}
\begin{cases}
    z \in C^{0,\a}([0,T];G^{[\frac{1}{\a}]}(\R^{N})),\\
    A \in C^{2}(\R^{N}\times\R;\R^{N}), \\
    a,b \in\Lip^{\g+2}(\R^{N}\times\R),\text{ for some } \gamma>\frac{1}{\a}\ge1\text{ and}\\
    b(x,0)=0   \text{ for all } x\in\R^{N}, 
\end{cases}
\end{equation}
and note that it has been shown in  \cite{GS14} that, under these assumptions,  the theory of pathwise entropy solutions to \eqref{eq:SCL-inhomo} is well posed.

Fix $\D t>0$, a partition $\{t_0,\dots,t_K\}$ of $[0,T]$ given by $t_{k}:=k\D t$ and let
\[
\D z:=\max_{k=1,\dots,N-1}|z_{t_{k+1}}-z_{t_{k}}|.
\]
The approximation scheme
is given by
\begin{equation}\left\{\begin{aligned}
&\partial_{t}f_{\D t}+\sum_{i=1}^{N}a^{i}(x,\xi)\partial_{x_{i}}f_{\D t}\circ dz^{i}+\sum_{i=1}^{N}\partial_{x_{i}}A^{i}(x,\xi)\partial_{\xi}f_{\D t}\circ dz^{i}  =0\quad\text{on }(t_{k},t_{k+1}), \label{eq:fdt_formal}\\
&f_{\D t}(x,\xi,t_{k})  =\chi(u_{\D t}(x,t_{k}),\xi), 
\end{aligned}\right.\end{equation}
where 
\begin{equation}\label{takis1001}
u_{\D t}(x,0):=u_{0}(x) \ \  \text{and} \ \ u_{\D t}(x,t):=\int f_{\D t}(x,\xi,t-)d\xi.
\end{equation}
We begin by expressing $f_{\D t}$ in terms of the characteristics of \eqref{eq:fdt_formal}. For each final time $t_{1}\ge0$, we consider the backward characteristics
\begin{equation*}\left\{\begin{aligned}
dX_{(x,\xi,t_{1})}^{i}(t) & =a_{i}(X_{(x,\xi,t_{1})}(t),\Xi_{(x,\xi,t_{1})}(t))dz^{t_{1},i}(t),\ X_{(x,\xi,t_{1})}^{i}(0)=x^{i},\ i=1,\dots,N,
\\
d\Xi_{(x,\xi,t_{1})}(t) & =-\sum_{i=1}^{N}(\partial_{x_{i}}A)(X_{(x,\xi,t_{1})}(t),\Xi_{(x,\xi,t_{1})}(t))dz^{t_{1},i}(t),\,\Xi_{(x,\xi,t_{1})}(0)=\xi,
\end{aligned}\right.\end{equation*}
where $z^{t_{1}}$ is the time-reversed rough path, that is, for $t\in[0,t_{1}]$,
\begin{equation}\label{eq:reversed_rp}
z^{t_{1}}(t):=z(t_{1}-t).
\end{equation}
Note that, in view of \eqref{eq:inhomo_assumptions}, the flow of backward characteristics $(x,\xi)\mapsto(X_{(x,\xi,t_{1})}(t),\Xi_{(x,\xi,t_{1})})$ is volume preserving on $\R^{N+1}$ and, in addition, for all $t_{1},t\in[0,T]$ and $(x,\xi)\in\R^{N+1}$,
\begin{align}
\sgn(\Xi_{(x,\xi,t_{1})}(t))  =\sgn(\xi)\label{eq:flow_sgn}\quad\text{ and }\quad
\Xi_{(x,0,t_{1})}(t)  =0.
\end{align}
Let
\[
(Y_{(x,\xi,t_{1})}(t),\z_{(x,\xi,t_{1})}(t)):=(X_{(x,\xi,t_{1})}(t),\Xi_{(x,\xi,t_{1})})^{-1}.
\]
The solution $f_{\D t}$ to \eqref{eq:fdt_formal}, for $t\in[t_{k},t_{k+1})$,  is given by
\[
f_{\D t}(x,\xi,t)=f_{\D t}\left(X_{(x,\xi,t)}(t-t_{k}),\Xi_{(x,\xi,t)}(t-t_{k}),t_{k}\right).
\]
We have:
\begin{thm}
\label{thm:x-dep}Let $u_{0}\in(L^{1}\cap L^{2})(\R^{N})$ and assume \eqref{eq:inhomo_assumptions}. Then, as $\D t\to0$,
\[
u_{\D t}\to u\quad\text{in }L^{1}(\R^{N}\times[0,T]).
\]
\end{thm}
\begin{proof}
We begin with a brief outline of the proof: Firstly, as in the proof of Theorem \ref{thm:main} we rewrite the scheme in a kinetic formulation with a defect measure $m_{\D t}$. Then we establish uniform in $\D t$ estimates for $f_{\D t}$ and $m_{\D t}$. This allows to extract weakly$\star$- convergent subsequences $f_{\D t}\overset{\star}\rightharpoonup f$, $m_{\D t}\overset{\star}\rightharpoonup m$. We then identify the limit $f$ as a pathwise generalized entropy solution to \eqref{eq:SCL-inhomo}. Since, in view of  Theorem 3.1] in \cite{GS14}, generalized entropy solutions are unique, it follows that $f=\chi$, and this  yields the weak convergence of the approximants $f_{\D t}$. In the last step we deduce strong convergence.

\textbf{\textit{Step 1: The kinetic formulation of the approximation scheme.}} Similarly to the homogeneous setting we observe that  the semi-discretization scheme has the following kinetic representation:
\begin{equation}\label{eq:approx_kinetic-1}
\partial_{t}f_{\D t}+\sum_{i=1}^{N}a^{i}(x,\xi)\partial_{x_{i}}f_{\D t}\circ dz^{i}+\sum_{i=1}^{N}\partial_{x_{i}}A^{i}(x,\xi)\partial_{\xi}f_{\D t}\circ dz^{i} =\partial_{\xi}m_{\D t}
\end{equation}
where
$$\partial_{\xi}m_{\D t}:= \sum_{k}\d(t-t_{k})(\mcM f_{\D t}-f_{\D t}),$$   
$m_{\D t}$ being a non-negative measure on $\R^N\times\R\times[0,T]$, and  $\mcM$ is defined as in \eqref{eq:mcM}. 

We pass to the stable form of \eqref{eq:approx_kinetic-1} by convolution along characteristics.  For any $\varrho^{0}\in C_{c}^{\infty}(\R^{N+1})$, $t_{0}\in[0,T]$ and $(y,\eta)\in\R^{N+1}$, we consider 
\begin{equation}
\vr_{t_{0}}(x,y,\xi,\eta,t):=\vr^{0}\left(\begin{array}{cc}
X_{(x,\xi,t)}(t-t_{0})-y\\
\Xi_{(x,\xi,t)}(t-t_{0})-\eta
\end{array}\right).\label{eq:transport_stable-1-1}
\end{equation}
Then, in the sense of distributions in $t\in[0,T]$,
\[
\partial_{t}(f_{\D t}\ast\vr_{t_{0}})(y,\eta,t)=-\int\partial_{\xi}\vr_{t_{0}}(x,y,\xi,\eta,t)m_{\D t}(x,\xi,t)dtdxd\xi,
\]
which is equivalent to
\begin{equation}
(f_{\D t}\ast\vr_{t_{0}})(y,\eta,t)-f_{\D t}\ast\vr_{t_{0}}(y,\eta,s)=-\int_{(s,t]}\int\partial_{\xi}\vr_{t_{0}}(x,y,\xi,\eta,t)m_{\D t}(x,\xi,t)dtdxd\xi.\label{eq:integrated_soln}
\end{equation}
for all $s<t$, $s,t\in[0,T].$

\textbf{\textit{Step 2: Stable apriori estimate.}} We establish uniform in $\D t$ estimates for $ $$f_{\D t}$ and $m_{\D t}$. We begin with an $L^1$-estimate.
\begin{lem}
\label{lem:energy-bounds} Let $u_{0}\in(L^{1}\cap L^{2})(\R^{N})$ and assume \eqref{eq:inhomo_assumptions}. Then, for all $t\in[0,T]$,
\begin{align}
\int|f_{\D t}|(x,\xi,t)dxd\xi\le & \|u_{0}\|_{1}.\label{eq:l1_bdd}
\end{align}
and, for some constant $M>0$ independent of $\D t$,
\begin{align}
\frac{1}{2}\int_{0}^{t}\int m_{\D t}(x,\xi,r)d\xi dxdr+\int f_{\D t}(x,\xi,t)\xi dxd\xi\le & \frac{1}{2}\|u_{0}\|_{2}^{2}+M\|u_{0}\|_{1}.\label{eq:energy_bound}
\end{align}
\end{lem}
\begin{proof}
Since $(x,\xi)\mapsto(X_{(x,\xi,t_{1})}(t),\Xi_{(x,\xi,t_{1})})$ is volume-preserving, using \eqref{eq:mcM-contraction} we find, for all $t\in[t_{k},t_{k+1})$,
\begin{align*}
\int|f_{\D t}|(x,\xi,t)dxd\xi & =\int|f_{\D t}|\left(X_{(x,\xi,t)}(t-t_{k}),\Xi_{(x,\xi,t)}(t-t_{k}),t_{k}\right)dxd\xi
  =\int|f_{\D t}|\left(x,\xi,t_{k}\right)dxd\xi\\
 & =\int|\mcM f_{\D t}|\left(x,\xi,t_{k}-\right)dxd\xi
  \le\int|f_{\D t}|\left(x,\xi,t_{k}-\right)dxd\xi,
\end{align*}
which proves \eqref{eq:l1_bdd} by iteration. 

Then \eqref{eq:energy_bound} follows as in Lemma 4.7  \cite[Lemma 4.7]{GS14}.
\end{proof}
Next we show that the approximants $f_{\D t}$ are uniformly tight.
\begin{lem}
\label{lem:tightness}Let $u_{0}\in(L^{1}\cap L^{2})(\R^{N})$  and assume \eqref{eq:inhomo_assumptions}. The family $f_{\D t}$ is uniformly tight, that is, for each $\ve>0$, there is an $R>0$ (independent of $\D t$) such that
\[
\sup_{t\in[0,T]}\int_{B_{R}^{c}\times\R}|f_{\D t}|(x,\xi,t)dxd\xi dt\le\ve.
\]
\end{lem}
\begin{proof}
Choose $\varrho^{s,0}C_{c}^{\infty}(\R^N)$ and $\varrho^{v,0}\in C_{c}^{\infty}(\R)$ and consider  \eqref{eq:transport_stable-1-1} with $\varrho^{0}(x,\xi):=\varrho^{s,0}(x)\varrho^{v,0}(\xi)$; the superscripts $s, v$ refer to the state and velocity variables respectively. 

Then
\begin{align*}
\partial_{\xi}\rho_{t_{0}}(x,0,\xi,0,t) & =\partial_{\xi}(\varrho^{s,0}(X_{(x,\xi,t)}(t-t_{0}))\varrho^{v,0}(\Xi_{(x,\xi,t)}(t-t_{0})))\\[1mm]
 & =(\partial_{\xi}\varrho^{s,0}(X_{(x,\xi,t)}(t-t_{0})))\varrho^{v,0}(\Xi_{(x,\xi,t)}(t-t_{0}))\\[1mm]
 & +\varrho^{s,0}(X_{(x,\xi,t)}(t-t_{0}))\partial_{\xi}(\varrho^{v,0}(\Xi_{(x,\xi,t)}(t-t_{0})))\\[1mm]
 & =D\varrho^{s,0}(X_{(x,\xi,t)}(t-t_{0}))\cdot(\partial_{\xi}X_{(x,\xi,t)}(t-t_{0}))\varrho^{v,0}(\Xi_{(x,\xi,t)}(t-t_{0}))\\[1mm]
 & +\varrho^{s,0}(X_{(x,\xi,t)}(t-t_{0}))D\varrho^{v,0}(\Xi_{(x,\xi,t)}(t-t_{0}))\partial_{\xi}\Xi_{(x,\xi,t)}(t-t_{0}).
\end{align*}
Fix $\ve>0$. It follows from Lemma \ref{lem:hoelder_rp} that we may choose 
$\d>0$, $s<t$, $t_{0}\in[s,t]$ and $|t-s|$ so small that
, for all $(x,\xi)\in\R^{N+1}$ and $r\in[s,t]$, 
\begin{align}
\partial_{\xi}\Xi_{(x,\xi,r)}(r-t_{0})  \ge0, \quad
|X_{(x,\xi,r)}(r-t_{0})-x|  \le\frac{1}{4}\label{eq:der-bdd}\quad\text{and}\quad
|\partial_{\xi}X_{(x,\xi,r)}(r-t_{0})|  \le\d. 
\end{align}
Hence, 
for all $(x,\xi)\in\R^{N+1}$, $r\in[s,t]$,
\begin{align*}
  &-D\varrho^{s,0}(X_{(x,\xi,r)}(r-t_{0}))\cdot(\partial_{\xi}X_{(x,\xi,r)}(r-t_{0}))\varrho^{v,0}(\Xi_{(x,\xi,r)}(r-t_{0}))\\[1mm]
  &\le |D\varrho^{s,0}(X_{(x,\xi,r)}(r-t_{0})||\partial_{\xi}X_{(x,\xi,r)}(r-t_{0})||\varrho^{v,0}(\Xi_{(x,\xi,r)}(r-t_{0}))|\\[1mm]
   & \le\d|D\varrho^{s,0}(X_{(x,\xi,r)}(r-t_{0}))|.
\end{align*}
Next we consider a sequence of $\vr^{v,0}_L$'s such that $\vr_L^{v,0}\to\sgn$ in $L^{\infty}(\R)$ for $L \to\infty$, $\vr_L ^{v,0}$ non-decreasing on $[-1,1]$ and $|D\vr_L ^{v,0}(\xi)|\le 1$ for all $1\le |\xi|$ and $D\vr_L ^{v,0}(\xi)=0$ for all $1 \le |\xi| \le L $. 
Then
\begin{align*}
  &-\varrho^{s,0}(X_{(x,\xi,r)}(r-t_{0}))D\varrho_L ^{v,0}(\Xi_{(x,\xi,r)}(r-t_{0}))\partial_{\xi}\Xi_{(x,\xi,r)}(r-t_{0})\\
  &\le -\varrho^{s,0}(X_{(x,\xi,r)}(r-t_{0}))D\varrho_L ^{v,0}(\Xi_{(x,\xi,r)}(r-t_{0}))\partial_{\xi}\Xi_{(x,\xi,r)}(r-t_{0})1_{|\Xi_{(x,\xi,r)}(r-t_{0})|\ge 1}.
\end{align*}
Using Lemma \ref{lem:energy-bounds} and dominated convergence, we conclude
\begin{align*}
  -\lim_{L \to \infty} \int_{(s,t]}\int \varrho^{s,0}(X_{(x,\xi,r)}(r-t_{0}))D\varrho_L ^{v,0}(\Xi_{(x,\xi,r)}(r-t_{0}))\partial_{\xi}\Xi_{(x,\xi,r)}(r-t_{0})m_{\D t}(x,\xi,r)dxd\xi dr
  \le 0
\end{align*}
and, hence,
\begin{align*}
    -\lim_{L \to \infty} \int_{(s,t]}\int \partial_{\xi}\rho_{t_{0},L}(x,0,\xi,0,r)m_{\D t}(x,\xi,r)dxd\xi dr
    \le  \int_{(s,t]}\int \d|D\varrho^{s,0}(X_{(x,\xi,r)}(r-t_{0})|m_{\D t}(x,\xi,r)dxd\xi dr.
\end{align*}
%
Thus, with $(y,\eta)=(0,0)\in\R^{N+1}$ in \eqref{eq:integrated_soln}, we get
\begin{align*}
 & \lim_{L \to\infty} \int f_{\D t}(x,\xi,t)\varrho_{t_{0},L }(x,0,\xi,0,t)dxd\xi-\lim_{L \to\infty}\int f_{\D t}(x,\xi,s)\varrho_{t_{0},L }(x,0,\xi,0,s)dxd\xi\\
 & =-\lim_{L \to\infty}\int_{(s,t]}\int\partial_{\xi}\varrho_{t_{0},L }(x,0,\xi,0,r)m_{\D t}(x,\xi,r)dxd\xi dr\\
 & \le\d\int_{(s,t]}\int|D\varrho^{s,0}(X_{(x,\xi,r)}(r-t_{0}))|m_{\D t}(x,\xi,r)dxd\xi dr.
\end{align*}
We choose $t_{0}=s$, use that $\vr_L ^{v,0}\to\sgn$ in $L^{\infty}(\R)$ for $L \to\infty$ and $\sgn(f_{\D t}(x,\xi,t))=\sgn(\xi)$ to find
\begin{equation}\label{eq:loc_l1}
\begin{cases}
  \int|f_{\D t}|(x,\xi,t)\varrho^{s,0}(X_{(x,\xi,t)}(t-s))dxd\xi-\int|f_{\D t}|(x,\xi,s)\varrho^{s,0}(x)dxd\xi\\[1mm]
\qquad  \le\d\int_{(s,t]}\int|D\varrho^{s,0}(X_{(x,\xi,r)}(r-s))|m_{\D t}(x,\xi,r)dxd\xi dr. 
\end{cases}
\end{equation}
Let $R>0$ large enough to be fixed later 
and choose $\vr^{s,0}:\R^N\to[0,1]$ such that
\[
\vr^{s,0}=\begin{cases}
1 &\quad |x|\ge R-\frac{1}{4},\\[1mm]
0 &\quad |x|<R-\frac{1}{2},
\end{cases}\quad\text{and}\quad |D\vr^{s,0}|\le4. 
\]
If follows, using \eqref{eq:der-bdd}, that
\begin{align*}
\varrho^{s,0}(X_{(x,\xi,t)}(t-s)) & =\begin{cases}
1 &\quad |x|\ge R,\\
0 &\quad |x|\le R-1.
\end{cases}
\end{align*}
We employ again Lemma \ref{lem:hoelder_rp} to choose a partition $0=\tau_{0}<\tau_{1}<\dots<\tau_{\td M}=T$ of $[0,T]$ with  $\td M=\td M(\d)$ such that \eqref{eq:der-bdd} is satisfied for all intervals $[s,t]=[\tau_{k},\tau_{k+1}]$. Then, using \eqref{eq:loc_l1}, for all $t\in[0,T]$ and $k\in\{0,\dots,\tilde M\}$ such that $t\in[\tau_{k},\tau_{k+1})$, we find
\begin{align*}
  \int_{B_{R}^{c}}|f_{\D t}|(x,\xi,t)dxd\xi
 & \le\int|f_{\D t}|(x,\xi,t)\varrho^{s,0}(X_{(x,\xi,t)}(t-\tau_{k}))dxd\xi\\
 & \le\int|f_{\D t}|(x,\xi,\tau_{k})\varrho^{s,0}(x)dxd\xi+\d\int_{(\tau_{k},t]}\int|D\varrho^{s,0}(X_{(x,\xi,r)}(r-\tau_{k}))|m_{\D t}dxd\xi dr\\
 & \le\int_{B_{R-1}^{c}}|f_{\D t}|(x,\xi,\tau_{k})dxd\xi+4\d\int_{(\tau_{k},t]}\int m_{\D t}dxd\xi dr,
\end{align*}
which, after an iteration and in view of Lemma \ref{lem:energy-bounds}, yields 
\begin{align*}
  \int_{B_{R}^{c}}|f_{\D t}|(x,\xi,t)dxd\xi
 & \le\int_{B_{R-\td M}^{c}}|f_{\D t}|(x,\xi,0)dxd\xi+4\d\int_{[0,t]}\int m_{\D t}dxd\xi dr\\
 & \le\int_{B_{R-\td M}^{c}}|u_{0}|(x)dx+4\d(\frac{1}{2}\|u_{0}\|_{2}^{2}+M\|u_{0}\|_{1}).
\end{align*}
To conclude, we first choose $\d<\frac{\ve}{2\|u_{0}\|_{2}^{2}+2M\|u_{0}\|_{1}}$ and then $R$ large enough.
\end{proof}
\textbf{\textit{Step 3: The weak convergence.}} For all $t_{0}\ge0$, all test functions $\vr_{t_{0}}$ given by \eqref{eq:transport_stable-1-1} with $\vr^{0}\in C_{c}^{\infty}$ and all $\vp\in C_{c}^{\infty}([0,T))$, we have
\begin{equation}\label{eq:approx_rough_kinetic}
\begin{cases}
  \int_{0}^{T}\partial_{t}\vp(r)(\vr_{t_{0}}\ast f_{\D t})(y,\eta,r)dr+\vp(0)(\vr_{t_{0}}\ast f_{\D t})(y,\eta,0)\\[1.2mm]
 \qquad  =\int_{0}^{T}\int\vp(r)\partial_{\xi}\vr_{t_{0}}(x,y,\xi,\eta,r)m{}_{\D t}(x,\xi,r)dxd\xi dr,
\end{cases}
\end{equation}
that is, 
\begin{equation}\label{eq:approx_rough_kinetic_2}
\begin{cases}
 \int_{0}^{T}\int\partial_{t}\vp(r)\vr_{t_{0}}(x,y,\xi,\eta,r)f_{\D t}(x,\xi,r)dxd\xi dr+\int\vp(0)\vr_{t_{0}}(x,y,\xi,\eta,0)\chi(u^{0}(x),\xi)dxd\xi\\[1.2mm] 
\qquad =\int_{0}^{T}\int\vp(r)\partial_{\xi}\vr_{t_{0}}(x,y,\xi,\eta,r)m_{\D t}(x,\xi,r)dxd\xi dr.
\end{cases}
\end{equation}
Moreover, once again using Lemma \ref{lem:hoelder_rp} we find that, for some $C>0$ and all $t\in [0,T]$,
\begin{equation}
\sup_{x,\xi}\left\Vert \left(\begin{array}{c}
X_{(x,\xi,t)}(t-\cdot)-x\\
\Xi_{(x,\xi,t)}(t-\cdot)-\xi
\end{array}\right)\right\Vert _{C^{0}([0,T])}\le C.\label{eq:unif_bound}
\end{equation}
Since $\vr^{0}$ has compact support so does $\vr_{t_{0}}$ in view of \eqref{eq:unif_bound}. Moreover, Lemma \ref{lem:energy-bounds} gives
\[
\sup_{t\in[0,T]}\|f_{\D t}(\cdot,\cdot,t)\|_{L^{1}(\R^{N}\times\R)}\le\|u_{0}\|_{1}.
\]
We use next Lemma \ref{lem:tightness} and $|f_{\D t}|\le1$ to find a subsequence (again denoted as $f_{\D t}$) such that, as $\D t\to 0$,
\begin{align*}
f_{\D t}  \overset{*}\rightharpoonup f\text{ in }L^{\infty}(\R^{N}\times\R\times[0,T])\quad\text{and}\quad
f_{\D t}  \rightharpoonup f\text{ in }L^{1}(\R^{N}\times\R\times[0,T]).
\end{align*}
Moreover, Lemma \ref{lem:energy-bounds} yields
\[
\|f\|_{L^{\infty}([0,T];L^{1}(\R^{N}\times\R))}\le\|u_{0}\|_{1}.
\]
Since $\sgn(f_{\D t}(x,\xi,t))=\sgn(\xi)$, the weak$\star$ convergence of the $f_{\D t}$'s implies
\[
f(x,\xi,t)\sgn(\xi)=|f|(x,\xi,t)\le1.
\]
Next, we note that
\begin{align*}
\partial_{\xi}f_{\D t}= & \sum_{k}\partial_{\xi}(\chi(X_{(x,\xi,t)}(t-t_{k}),\Xi_{(x,\xi,t)}(t-t_{k}),t_{k}))1_{[t_{k},t_{k+1})}(t)\\
= & \sum_{k}(\partial_{\xi}\chi)(X_{(x,\xi,t)}(t-t_{k}),\Xi_{(x,\xi,t)}(t-t_{k}),t_{k})\partial_{\xi}\Xi_{(x,\xi,t)}(t-t_{k})1_{[t_{k},t_{k+1})}(t)\\
 & +\sum_{k}(D_{x}\chi)(X_{(x,\xi,t)}(t-t_{k}),\Xi_{(x,\xi,t)}(t-t_{k}),t_{k})\cdot\partial_{\xi}X_{(x,\xi,t)}(t-t_{k})1_{[t_{k},t_{k+1})}(t).
\end{align*}
Moreover, \eqref{eq:flow_sgn} implies that, in the sense of distributions,
\begin{align}
 & (\partial_{\xi}\chi)(X_{(x,\xi,t)}(t-t_{k}),\Xi_{(x,\xi,t)}(t-t_{k}),t_{k})\nonumber \\[1mm]
 & =\d(\Xi_{(x,\xi,t)}(t-t_{k}))-\d\big(\Xi_{(x,\xi,t)}(t-t_{k})-u(X_{(x,\xi,t)}(t-t_{k}),t_{k})\big)\label{eq:chi_der}\\[1mm]
 & =\d(\xi)-\d\big(\Xi_{(x,\xi,t)}(t-t_{k})-u(X_{(x,\xi,t)}(t-t_{k}),t_{k})\big),\nonumber 
\end{align}
where, for $\vp\in C_{c}^{\infty}(\R^{N+1})$, 
\[
\d(\Xi_{(x,\xi,t)}(t-t_{k})-u(X_{(x,\xi,t)}(t-t_{k}),t_{k}))(\vp):=\int\vp(Y_{(x,\xi,t_{k})}(t),\z_{(x,\xi,t_{k})}(t))\d(\xi-u(x,t_{k}))dxd\xi,
\]
and thus
\begin{equation}\label{eq:fdt_xi_der}
\begin{cases}
\partial_{\xi}f_{\D t}=  \d(\xi)-\nu_{\D t}(x,\xi,t) +\sum_{k}\d(\xi)(\partial_{\xi}\Xi_{(x,\xi,t)}(t-t_{k})-1)1_{[t_{k},t_{k+1})}(t) \\[1mm]
\qquad +\sum_{k}D_{x}\chi(X_{(x,\xi,t)}(t-t_{k}),\Xi_{(x,\xi,t)}(t-t_{k}),t_{k})\cdot\partial_{\xi}X_{(x,\xi,t)}(t-t_{k})1_{[t_{k},t_{k+1})}(t),
\end{cases}
\end{equation}
with
\[
\nu_{\D t}(x,\xi,t):=\sum_{k}\d(\Xi_{(x,\xi,t)}(t-t_{k})-u(X_{(x,\xi,t)}(t-t_{k}),t_{k}))\partial_{\xi}\Xi_{(x,\xi,t)}(t-t_{k})1_{[t_{k},t_{k+1})}(t) .
\]
We use again Lemma \ref{lem:hoelder_rp} to get for $\D t$ small enough and all $t\in [t_{k},t_{k+1})$,
\begin{equation}
\partial_{\xi}\Xi_{(x,\xi,t)}(t-t_{k})\in[0,2],\label{eq:sign}
\end{equation}
which implies that $\nu_{\D t}$ is a non-negative measure. 

Furthermore, for all $R>0$, \eqref{eq:sign} and \eqref{eq:unif_bound} give, for some constants $\td R$, $C>0$ independent of $\D t$,
\begin{align*}
\int_{0}^{T}\int_{B_{R}}\nu_{\D t}dxd\xi dt & =\sum_{k}\int_{t_{k}}^{t_{k+1}}\int\int_{B_{R}}\d(\Xi_{(x,\xi,t)}(t-t_{k})-u(X_{(x,\xi,t)}(t-t_{k}),t_{k}))\partial_{\xi}\Xi_{(x,\xi,t)}(t-t_{k})dxd\xi dt\\
 & =\sum_{k}\int_{t_{k}}^{t_{k+1}}\int\int_{B_{\td R}}\d(\xi-u(x,t_{k}))\partial_{\xi}\Xi_{(x,\xi,t)}(t-t_{k})|_{Y_{(x,\xi,t_{k})}(t),\z_{(x,\xi,t_{k})}(t)}dxd\xi dt\\
 & \le2\sum_{k}\int_{t_{k}}^{t_{k+1}}\int\int_{B_{\td R}}\d(\xi-u(x,t_{k}))dxd\xi dt \le C.
\end{align*}
Hence, there exists a non-negative measure $\nu$ so that, along a subsequence,
\[
\nu_{\D t}\overset{*}\rightharpoonup\nu .
\]
Observe that, for each $\vp\in C_{c}^{\infty}(\R^{N}\times\R\times[0,T])$,
\begin{align*}
 & \sum_{k}\int\vp(x,\xi,t)(D_{x}\chi)(X_{(x,\xi,t)}(t-t_{k}),\Xi_{(x,\xi,t)}(t-t_{k}),t_{k})\cdot\partial_{\xi}X_{(x,\xi,t)}(t-t_{k})1_{[t_{k},t_{k+1})}(t)dxd\xi dt\\
 & =\sum_{k}\int_{t_{i}}^{t_{i+1}}\int(D_{x}\chi)(x,\xi,t_{k})\cdot\vp(Y_{(x,\xi,t_{k})}(t),\z_{(x,\xi,t_{k})}(t),t)\partial_{\xi}X_{(x,\xi,t)}(t-t_{k})|_{Y_{(x,\xi,t_{k})}(t),\z_{(x,\xi,t_{k})}(t)}dxd\xi dt\\
 & =-\sum_{k}\int_{t_{i}}^{t_{i+1}}\int\chi(x,\xi,t_{k})\cdot D_{x}\left(\vp(Y_{(x,\xi,t_{k})}(t),\z_{(x,\xi,t_{k})}(t),t)\partial_{\xi}X_{(x,\xi,t)}(t-t_{k})|_{Y_{(x,\xi,t_{k})}(t),\z_{(x,\xi,t_{k})}(t)}\right)dxd\xi dt,
\end{align*}
and, since Lemma \ref{lem:hoelder_rp} yields that, as $\D t\to 0$,
\[
\sup_{t\in[t_{k},t_{k+1}]}\|\partial_{\xi}X_{(\cdot,\cdot,t)}(t-t_{k})\|_{C^{1}(\R^{N+1})}\to0,
\]
we find
\begin{align*}
 & \sum_{k}\int\vp(x,\xi,t)(D_{x}\chi)(X_{(x,\xi,t)}(t-t_{k}),\Xi_{(x,\xi,t)}(t-t_{k}),t_{k})\cdot\partial_{\xi}X_{(x,\xi,t)}(t-t_{k})1_{[t_{k},t_{k+1})}(t)dxd\xi dt\to0.
\end{align*}
Moreover, again Lemma \ref{lem:hoelder_rp} gives that, for $\D t\to0$,
\[
\|\partial_{\xi}\Xi_{(\cdot,\cdot,t)}(t-t_{k})-1\|_{C(\R^{N+1})}\to0,
\] 
and thus letting $\D t\to0$ in \eqref{eq:fdt_xi_der} we find that, in the sense of distributions,
\begin{align*}
\partial_{\xi}f= & \d(\xi)-\nu.
\end{align*}

Recall that (see Lemma \ref{lem:energy-bounds}), for all $t\in[0,T]$
\begin{align*}
\frac{1}{2}\int_{0}^{t}\int m_{\D t}(x,\xi,r)d\xi dxdr\le & \frac{1}{2}\|u_{0}\|_{2}^{2}+M\|u_{0}\|_{1}.
\end{align*}
It follows that  there exists some nonnegative measure $m$ and a weak$\star$ convergent subsequence such that $m_{\D t}\overset{*}\rightharpoonup m$. 

Taking the limit in \eqref{eq:approx_rough_kinetic_2} then yields
\begin{align*}
 & \int_{0}^{T}\int\partial_{t}\vp(r)\vr_{t_{0}}(x,y,\xi,\eta,r)f(x,\xi,r)dxd\xi dr+\int\vp(0)\vr_{t_{0}}(x,y,\xi,\eta,0)\chi(u^{0}(x),\xi)dxd\xi\\
 & =\int_{0}^{T}\int\vp(r)\partial_{\xi}\vr_{t_{0}}(x,y,\xi,\eta,r)m(x,\xi,r)dxd\xi dr.
\end{align*}
Hence, $f$ is a generalized rough kinetic solution to \eqref{eq:SCL-inhomo}. The uniqueness of generalized rough kinetic solutions (see \cite[Theorem 3.1]{GS14}) yields that $f=\chi$ and thus $f$ is the unique pathwise entropy solution to \eqref{eq:SCL-inhomo}. Hence, the whole sequence $f_{\D t}$ converges to $\chi$ weakly$\star$ in $L^{\infty}(\R^{N}\times\R\times[0,T])$ and weakly in $L^{1}(\R^{N}\times\R\times[0,T])$.

\textbf{\textit{Step 4: The strong convergence.}} We note that, in view of the weak convergence of $f_{\D t}$ to $\chi$ in $L^{1}(\R^{N}\times\R\times[0,T])$, we have, for $\D t\to 0$, 
\begin{align*}
\int_{0}^{T}\int|f_{\D t}-\chi|^{2}dxd\xi dt & =\int_{0}^{T}\int|f_{\D t}|^{2}-2f_{\D t}\chi+|\chi|^{2}dxd\xi dt
  \le\int_{0}^{T}\int|f_{\D t}|-2f_{\D t}\chi+|\chi|dxd\xi dt\\
 & =\int_{0}^{T}\int f_{\D t}\sgn(\xi)-2f_{\D t}\chi+|\chi|dxd\xi dt \to\int_{0}^{T}\int\chi\sgn(\xi)-2\chi\chi+|\chi|dxd\xi dt \\
  &=0.
\end{align*}
The uniform tightness of $f_{\D t}$ then implies $\int_{0}^{T}\int|f_{\D t}-\chi|dxd\xi dt\to0$ and, hence, as $\D t\to0$,
\begin{align*}
\int_{0}^{T}\int|u_{\D t}-u|dxdt & =\int_{0}^{T}\int|\int f_{\D t}d\xi-\int\chi d\xi|dxdt
  \le\int_{0}^{T}\int|f_{\D t}-\chi|d\xi dxdt \to0.
\end{align*}
\end{proof}

\appendix

\section{Definitions and some estimates from the theory of rough paths\label{sec:RP}}

We briefly recall some basic facts of the Lyons' rough paths theory used in this paper. For more details we refer to Lyons and Qian \cite{L02} and Friz and Victoir \cite{FV10}. 

Given $x\in C^{1-\text{var}}([0,T];\R^{N})$, the space of continuous paths of bounded variation, the step $M$ signature $S_{M}(x)_{0,T}$ given by
\[
S_{M}(x)_{0,T}:=\left(1,\int_{0<u<T}dx_{u},\dots,\int_{0<u_{1}<\dots<u_{M}<T}dx_{u_{1}}\otimes\dots\otimes dx_{u_{M}}\right),
\]
takes values in the truncated step-$M$ tensor algebra 
\[
T^{M}(\R^{N})=\R\oplus\R^{N}\oplus(\R^{N}\otimes\R^{N})\oplus\ldots\oplus(\R^{N})^{\otimes M};
\]
in fact, $S_{M}(x)$ takes values in the smaller set $G^{M}(\R^{N})\subset T^{M}(\R^{N})$ given by 
\[
G^{M}(\R^{N}):=\left\{ S_{M}(x)_{0,1}:\ x\in C^{1-\text{var}}([0,1];\R^{N})\right\} .
\]
The Carnot-Caratheodory norm of $G^{M}(\R^{N})$ given by
\[
\|g\|:=\inf\left\{ \ \int_{0}^{1}|d\g|\ :\gamma\in C^{1-\text{var}}([0,1];\R^{N})\text{ and }S_{M}(\g)_{0,1}=g\right\} ,
\]
gives rise to a homogeneous metric on $G^{M}(\R^{N})$. 

Alternatively, for any $g\in T^{M}(\R^{N})$, we may set 
\[
|g|:=|g|_{T^{M}(\R^{N})}:=\max_{k=1\dots M}|\pi_{k}(g)|,
\]
where $\pi_{k}$ is the projection of $g$ onto the $k$-th tensor level, which is an inhomogeneous metric on $G^{M}(\R^{N})$. It turns out that the topologies induced by $\|\cdot\|$ and $|\cdot|$ are equivalent. 

For paths in $T^{M}(\R^{N})$ starting at the fixed point $e:=1+0+\ldots+0$ and $\beta\in(0,1]$, it is possible to define $\beta$-Hölder metrics extending the usual metrics for paths in $\R^{N}$ starting at zero. The homogeneous $\beta$-Hölder metric is denoted by $\ensuremath{d_{\beta-\text{Höl}}}$ and the inhomogeneous one by $\ensuremath{\rho_{\beta-\text{Höl}}}$. A corresponding norm is defined by $\ensuremath{\|\cdot\|_{\beta-\text{Höl}}=d_{\beta-\text{Höl}}(\cdot,0)}$, where $0$ denotes the constant $e$-valued path. 

A geometric $\beta$-Hölder rough path ${x}$ is a path in $T^{\lfloor1/\beta\rfloor}(\R^{N})$ which can be approximated by lifts of smooth paths in the $\ensuremath{d_{\beta-\text{Höl}}}$ metric. It can be shown that rough paths actually take values in $G^{\lfloor1/\beta\rfloor}(\R^{N})$. The space of geometric $\beta$-Hölder rough paths is denoted by $C^{0,\b}([0,T];G^{[\frac{1}{\b}]}(\R^{N}))$. 

We state next a  basic stability estimate for solutions to rough differential equations (RDE) of the form 
\[
dx=V(x)\circ d{z},
\]
where ${z}$ is a geometric $\a$-Hölder rough path. 

It is well known (see, for example, \cite{FV10}) that the RDE above has a flow $\psi^{{z}}$ of solutions. The following is taken from Crisan, Diehl, Friz and Oberhauser \cite[Lemma 13]{CDFO13}.
\begin{lem}
\label{lem:hoelder_rp}Let $\a\in(0,1)$, $\g>\frac{1}{\a}\ge1$, $k\in\N$ and assume that $V\in\Lip^{\g+k}(\R^{N};\R^{N}).$ For all $R>0$ there exist $C=C(R,\|V\|_{\Lip^{\g+k}})$ and $K=K(R,\|V\|_{\Lip^{\g+k}})$, which are non-decreasing in all arguments, such that, for all geometric $\a$-Hölder rough paths ${z}^{1},{z}^{2}\in C^{0,\a}([0,T];G^{[\frac{1}{\a}]}(\R^{N}))$ with $\|{z}^{1}\|_{\a-\Hoel;[0,T]}$, $\|{z}^{2}\|_{\a-\Hoel;[0,T]}\le R$ and all $n\in\{0,\dots,k\}$, 
\begin{align*}
\sup_{x\in\R^{N}}\|D^{n}(\psi^{{z}^{1}}-\psi^{{z}^{2}})(x)\|_{\a-\Hoel;[0,T]} & \le C\rho_{\a-\Hoel;[0,T]}({z}^{1},{z}^{2}),\\
\sup_{x\in\R^{N}}\|D^{n}((\psi^{{z}^{1}})^{-1}-(\psi^{{z}^{2}})^{-1})(x)\|_{\a-\Hoel;[0,T]} & \le C\rho_{\a-\Hoel;[0,T]}({z}^{1},{z}^{2})
\end{align*}
and, for all $n\in\{1,\dots,k\}$, 
\begin{align*}
\sup_{x\in\R^{N}}\|D^{n}\psi^{{z}^{1}}(x)\|_{\a-\Hoel;[0,T]}\le K\ \text{ and}\ \sup_{x\in\R^{N}}\|D^{n}(\psi^{{z}^{1}})^{-1}(x)\|_{\a-\Hoel;[0,T]}\le K.
\end{align*}

\end{lem}

\section{Pathwise entropy solutions to stochastic scalar conservation laws\label{sec:recall_e_soln}}

Assume \eqref{eq:homo_assumptions} and consider the spatially homogeneous problem
\begin{equation}
\begin{cases}
du+\sum_{i=1}^{N}\partial_{x_{i}}A^{i}(u)\circ dz^{i} & =0\quad\text{in }\R^{N}\times(0,T),\label{eq:SCL-app}\\[1mm]
u(\cdot, 0)  =u_{0}\in(L^{1}\cap L^{\infty})(\R^{N}) .
\end{cases}
\end{equation}
The following notion of pathwise entropy solutions to \eqref{eq:SCL-app} and its well-posedness were introduced in \cite{LPS13}.
\begin{defn}
A function $u\in(L^{1}\cap L^{\infty})(\R^{N}\times[0,T])$ is a pathwise entropy solution to \eqref{eq:SCL-app}, if there exists a nonnegative, bounded measure $m$ on $\R^{N}\times\R\times[0,T]$ such that, for all $\vr^{0}\in C_{c}^{\infty}(\R^{N+1})$, all $\vr$ given by
\[
\vr(x,y,\xi,\eta,t):=\vr^{0}(y-x+a(\xi)z(t),\xi-\eta),
\]
and all $\vp\in C_{c}^{\infty}([0,T))$,
\begin{align*}
 & \int_{0}^{T}\partial_{t}\vp(r)(\vr\ast\chi)(y,\eta,r)dr+\vp(0)(\vr\ast\chi)(y,\eta,0)
  =\int_{0}^{T}\int\vp(r)\partial_{\xi}\vr(x,y,\xi,\eta,r)m(x,\xi,r)dxd\xi dr,
\end{align*}
where the convolution along characteristics $\vr\ast\chi$ is defined by
\[
\vr\ast\chi(y,\eta,r):=\int\vr(x,y,\xi,\eta,r)\chi(x,\xi,r)dxd\xi.
\]
\end{defn}

The following is proved in \cite{LPS13}.
\begin{thm}
Let $u_{0}\in(L^{1}\cap L^{\infty})(\R^{N})$ and assume \eqref{eq:homo_assumptions}. Then there exists a unique pathwise entropy solution $u\in C([0,T];L^{1}(\R^{N}))$ satisfying, for all $p\in[1,\infty]$,
\[
\sup_{t\in[0,T]}\|u(t)\|_{p}\le\|u_{0}\|_{p},
\]
and
\[
\int_{0}^{T}\int_{[-\|u_{0}\|_{\infty},\|u_{0}\|_{\infty}]^{c}}\int_{\R^{N}}m(x,\xi,t)dxd\xi dt=0,\quad \int_{0}^{T}\int_{\R^{N+1}}m(x,\xi,t)dxd\xi dt\le\frac{1}{2}\|u_{0}\|_{2}^{2}.
\]
\end{thm}
The notion of pathwise entropy solutions was extended in \cite{LPS14} and \cite{GS14} to inhomogeneous stochastic scalar conservation laws of the type
\begin{equation}\label{eq:SCL-inhomo-app}
\begin{cases}
\partial_{t}u+\sum_{i=1}^{N}\partial_{x_{i}}A^{i}(x,u)\circ dz^{i}  =0\quad\text{in }\R^{N}\times(0,T),\\[1mm] 
u(\cdot, 0)  =u_{0}\in(L^{1}\cap L^{2})(\R^{N}). 
\end{cases}
\end{equation}
Assume that $A,z$ satisfy \eqref{eq:inhomo_assumptions}. 
For each $t_{1}\ge0$ and for $i=1,\dots,N,$ consider the backward characteristics 
\begin{equation*}\left\{\begin{aligned}
d{X}_{(x,\xi,t_{1})}^{i}(t) & =a^{i}(X_{(x,\xi,t_{1})}(t),\Xi_{(x,\xi,t_{1})}(t))\circ d{z}^{t_{1},i}(t),\\
d{\Xi}_{(x,\xi,t_{1})}(t) & =-\sum_{i=1}^{N}(\partial_{x_{i}}A^{i})(X_{(x,\xi,t_{1})}(t),\Xi_{(x,\xi,t_{1})}(t))\circ d{z}^{t_{1},i}(t),\\
X_{(x,\xi,t_{1})}^{i}(0) & =x^{i}\ \text{and}\ \Xi_{(x,\xi,t_{1})}(0)=\xi ,
\end{aligned}\right.\end{equation*}
where, for $t\in[0,t_{1}]$, ${z}^{t_{1}}$ is the time-reversed rough path defined in \eqref{eq:reversed_rp}.

Let $\vr_{t_{0}}$ be a test-function transported along the characteristics, that is, for some $\varrho^{0}\in C_{c}^{\infty}(\R^{N+1})$, $t_{0}\in[0,T]$, $(y,\eta)\in\R^{N+1}$,
\begin{equation}
\vr_{t_{0}}(x,y,\xi,\eta,t):=\vr^{0}\left(\begin{array}{cc}
X_{(x,\xi,t)}(t-t_{0})-y\\
\Xi_{(x,\xi,t)}(t-t_{0})-\eta
\end{array}\right).\label{eq:transport_stable-1}
\end{equation}
The following definition is Definition 2.1 and Definition 4.2 of \cite{GS14}.
\begin{defn}
\label{def:path_e-soln}Let $u_{0}\in(L^{1}\cap L^{2})(\R^{N})$. 
(i). A function $u\in L^{\infty}([0,T];L^{1}(\R^{N}))$ is a pathwise entropy solution to \eqref{eq:SCL-inhomo-app}, if there exists a nonnegative bounded measure $m$ on $\R^{N}\times\R\times[0,T]$ such that, for all $t_{0}\ge0$, all test functions $\vr_{t_{0}}$ given by \eqref{eq:transport_stable-1} with $\vr^{0}\in C_{c}^{\infty}$ and $\vp\in C_{c}^{\infty}([0,T)),$ 
\begin{equation}\label{eq:soln_prop}
\begin{cases}
  \int_{0}^{T}\partial_{t}\vp(r)(\vr_{t_{0}}\ast\chi)(y,\eta,r)dr+\vp(0)(\vr_{t_{0}}\ast\chi)(y,\eta,0)\\[1mm]
 \qquad  =\int_{0}^{T}\int\vp(r)\partial_{\xi}\vr_{t_{0}}(x,y,\xi,\eta,r)m(x,\xi,r)dxd\xi dr.
\end{cases}
\end{equation}
(ii). A function $f\in L^{\infty}([0,T];L^{1}(\R^{N}\times\R))$ is a generalized pathwise entropy solution to \eqref{eq:SCL-inhomo-app}, if there exists a nonnegative measure $\nu$ and a nonnegative, bounded measure $m$ on $\R^{N}\times\R\times[0,T]$ such that 
\begin{equation}
f(x,\xi,0)=\chi(u_{0}(x),\xi),\ |f|(x,\xi,t)=\sgn(\xi)f(x,\xi,t)\le1\ \text{ and }\ \frac{\partial f}{\partial\xi}=\d(\xi)-\nu(x,\xi,t),\label{eq:gen_kinetic_measure}
\end{equation}
and \eqref{eq:soln_prop} holds with $f$ replacing $\chi$, for all $t_{0}\ge0,$ $\vr_{t_{0}}$ as in \eqref{eq:transport_stable-1} and $\vp\in C_{c}^{\infty}([0,T))$.
\end{defn}
The following well-posedness results was proved in \cite{GS14}
\begin{thm}
Let $u_{0}\in(L^{1}\cap L^{2})(\R^{N})$ and assume \eqref{eq:inhomo_assumptions}. Then there exists a unique pathwise entropy solution to \eqref{eq:SCL-inhomo-app} and generalized pathwise entropy solutions to \eqref{eq:SCL-inhomo-app} are unique.
\end{thm}

\section{Indicator functions of BV functions\label{sec:indicator}}

We present here an observation which connects the $BV$-norms of $u$ and $\chi(u(\cdot),\xi)$.

\begin{lem}
\label{lem:BV-indicator}Let $u\in L_{loc}^{1}(\R^{N})$. Then
\begin{align*}
\|u\|_{BV} & =\int_{\R}\|\chi(u(\cdot),\xi)\|_{BV}d\xi.
\end{align*}
\end{lem}
\begin{proof}
It follows from Theorem 1 in Fleming and Rishel \cite{FR60} that, for any $u\in L_{loc}^{1}$,
\[
\|u\|_{BV}=\int_{\R}\|1_{(-\infty,u(\cdot))}(\xi)\|_{BV}d\xi.
\]
Hence, 
\begin{align*}
\|u\|_{BV} & =\|u^{+}\|_{BV}+\|u^{-}\|_{BV}\\
 & =\int_{\R}\|1_{(-\infty,u^{+}(\cdot))}(\xi)\|_{BV}d\xi+\int_{\R}\|1_{(-\infty,u^{-}(\cdot))}(\xi)\|_{BV}d\xi\\
 & =\int_{0}^{\infty}\|1_{(-\infty,u^{+}(\cdot))}(\xi)\|_{BV}d\xi+\int_{0}^{\infty}\|1_{(-\infty,u^{-}(\cdot))}(\xi)\|_{BV}d\xi\\
 & =\int_{0}^{\infty}\|1_{(0,u^{+}(\cdot))}(\xi)\|_{BV}d\xi+\int_{0}^{\infty}\|1_{(0,u^{-}(\cdot))}(\xi)\|_{BV}d\xi\\
 & =\int_{0}^{\infty}\|\chi(u^{+}(\cdot),\xi)\|_{BV}d\xi+\int_{0}^{\infty}\|\chi(u^{-}(\cdot),\xi)\|_{BV}d\xi.
\end{align*}
Since, for $\xi\ge0$, $\chi(u,\xi)=1_{(0,u)}(\xi)=1_{(0,u^{+})}(\xi)=\chi(u^{+},\xi)$ and $\chi(u,\xi)=-\chi(-u,-\xi)$ we get
\begin{align*}
\int_{\R}\|\chi(u(\cdot),\xi)\|_{BV}d\xi & =\int_{0}^{\infty}\|\chi(u(\cdot),\xi)\|_{BV}d\xi+\int_{-\infty}^{0}\|\chi(u(\cdot),\xi)\|_{BV}d\xi\\
 & =\int_{0}^{\infty}\|\chi(u^{+}(\cdot),\xi)\|_{BV}d\xi+\int_{-\infty}^{0}\|-\chi(-u(\cdot),-\xi)\|_{BV}d\xi\\
 & =\int_{0}^{\infty}\|\chi(u^{+}(\cdot),\xi)\|_{BV}d\xi+\int_{0}^{\infty}\|\chi(-u(\cdot),\xi)\|_{BV}d\xi\\
 & =\int_{0}^{\infty}\|\chi(u^{+}(\cdot),\xi)\|_{BV}d\xi+\int_{0}^{\infty}\|\chi(u^{-}(\cdot),\xi)\|_{BV}d\xi,
\end{align*}
and, hence, the claim.
%
\end{proof}

\subsection*{Acknowledgements}

Gess was supported by the research project Random dynamical systems and regularization by noise for stochastic partial differential equations funded by the German Research Foundation. Part of the work was completed while Gess and Perthame were visiting the University of Chicago. Souganidis is supported by the NSF grant DMS-1266383. 

\bibliographystyle{acm}
\bibliography{refs}

\end{document}